\documentclass[11pt]{article}

\title{Irredundant bases for the symmetric group}

\author{Colva M. Roney-Dougal and Peiran Wu}

\usepackage[a4paper,margin=2.54cm]{geometry}

\usepackage[T1]{fontenc}
\usepackage{lmodern}
\usepackage[UKenglish]{babel}

\usepackage[colorlinks]{hyperref}

\usepackage{chngcntr}

\usepackage{enumitem}

\usepackage{amsmath}
\usepackage{amssymb}
\usepackage{amsthm}
\usepackage{mathtools}
\usepackage{stmaryrd}

\usepackage{cleveref}
\Crefname{section}{\S}{\S\S}

\newtheorem{theorem}{Theorem}[section]
\newtheorem{lemma}[theorem]{Lemma}
\newtheorem{corollary}[theorem]{Corollary}
\theoremstyle{definition}
\newtheorem{remark}[theorem]{Remark}

\usepackage{foreign}
\newcommand*{\ie}{\foreign{i.e.}}

\newcommand*{\definedterm}[1]{\emph{#1}}

\DeclareMathOperator{\fix}{fix}
\DeclareMathOperator{\supp}{supp}
\DeclareMathOperator{\id}{id}
\DeclareMathOperator{\soc}{soc}

\newcommand*{\finitefield}[1]{\mathbb{F}_{#1}}
\newcommand*{\positivereal}{\mathbb{R}_{> 0}}

\newcommand*{\generatedby}[1]{\left\langle #1 \right\rangle}
\newcommand*{\cycleoneto}[1]{(1~2~\cdots~#1)}
\newcommand*{\cardinality}[1]{\left\lvert #1 \right\rvert}

\newcommand*{\symgroup}[1]{\operatorname{S}_{#1}}
\newcommand*{\altgroup}[1]{\operatorname{A}_{#1}}
\newcommand*{\symon}[1]{\operatorname{Sym}(#1)}
\newcommand*{\alton}[1]{\operatorname{Alt}(#1)}
\newcommand*{\agl}[2]{\operatorname{AGL}_{#1}(#2)}
\newcommand*{\aglon}[1]{\operatorname{AGL}(#1)}
\newcommand*{\glon}[1]{\operatorname{GL}(#1)}

\newcommand*{\normaliser}[2]{\operatorname{N}_{#1}(#2)}
\newcommand*{\centraliser}[2]{\operatorname{C}_{#1}(#2)}

\newcommand*{\basesize}[2]{\operatorname{b}(#1, #2)}
\newcommand*{\height}[2]{\operatorname{H}(#1, #2)}
\newcommand*{\mibs}[2]{\operatorname{I}(#1, #2)}
\newcommand*{\length}[1]{\ell(#1)}
\newcommand*{\relcomp}[2]{\operatorname{RC}(#1, #2)}
\newcommand*{\orderof}[1]{\cardinality{#1}}

\newcommand*{\hamming}[2]{d(#1, #2)}

\newcommand*{\vv}{\mathbf{v}}
\newcommand*{\vw}{\mathbf{w}}
\newcommand*{\vu}{\mathbf{u}}
\newcommand*{\vb}{\mathbf{b}}
\newcommand*{\vzero}{\mathbf{0}}

\begin{document}

\maketitle

\begin{abstract}
    \noindent
    An irredundant base of a group $G$ acting faithfully on a finite set $\Gamma$ is a sequence of points in $\Gamma$ that produces a strictly descending chain of pointwise stabiliser subgroups in $G$, terminating at the trivial subgroup.
    Suppose that $G$ is $\symgroup{n}$ or $\altgroup{n}$ acting primitively on $\Gamma$, and that the point stabiliser is primitive in its natural action on $n$ points.
    We prove that the maximum size of an irredundant base of $G$ is $O\left(\sqrt{n}\right)$, and in most cases $O\left((\log n)^2\right)$.
    We also show that these bounds are best possible.
\end{abstract}

\medskip


\counterwithout{theorem}{section}

\section{Introduction}

\noindent
Let $G$ be a finite group that acts faithfully and transitively on a set $\Gamma$ with point stabiliser $H$.
A sequence $(\gamma_1, \ldots, \gamma_l)$ of points of $\Gamma$ is an \definedterm{irredundant base} for the action of $G$ on $\Gamma$ if
\begin{equation}
    \label{equation:irredundant-base}
    G > G_{\gamma_1} > G_{\gamma_1, \gamma_2} > \cdots > G_{\gamma_1, \ldots, \gamma_{l}} = 1.
\end{equation}
Let $\basesize{G}{H}$ and $\mibs{G}{H}$ denote the minimum and the maximum sizes of an irredundant base in $\Gamma$ for $G$ respectively.

Recently, Gill \& Liebeck showed in \cite{gill_liebeck_2023} that if $G$ is an almost simple group of Lie type of rank $r$ over the field $\mathbb{F}_{p^f}$ of characteristic $p$ and $G$ is acting primitively, then
\[ \mibs{G}{H} \leqslant 177 r^8 + \Omega(f), \]
where $\Omega(f)$ is the number of prime factors of $f$, counted with multiplicity.

Suppose now that $G$ is the symmetric group $\symgroup{n}$ or the alternating group $\altgroup{n}$.
An upper bound for $\mibs{G}{H}$ is the maximum length of a strictly descending chain of subgroups in $G$, known as the \definedterm{length}, $\length{G}$, of $G$.
Define $\varepsilon(G) \coloneqq \length{G / \soc G}$.
Cameron, Solomon, and Turull proved in \cite{cameron_solomon_turull_1989} that
\begin{equation}
    \label{equation:length-symmetric-group}
    \nonumber
    \length{G} = \left\lfloor \frac{3n - 3}{2} \right\rfloor - b_n + \varepsilon(G),
\end{equation}
where $b_n$ denotes the number of $1$s in the binary representation of $n$.
For $n \geqslant 2$, this gives
\begin{equation}
    \label{equation:length-symmetric-group-inequality}
    \length{G} \leqslant \frac{3}{2} n - 3 + \varepsilon(G).
\end{equation}
This type of upper bound is best possible for such $G$ in general,
in that for the natural action of $\symgroup{n}$ or $\altgroup{n}$ on $n$ points, the maximum irredundant base size is $n - 2 + \varepsilon(G)$.
A recent paper \cite{gill_lodà_2022} by Gill \& Lodà determined the exact values of $\mibs{G}{H}$ when $H$ is maximal and intransitive in its natural action on $n$ points, and in each case $\mibs{G}{H} \geqslant n - 3 + \varepsilon(G)$.

In this article, we present improved upper bounds for $\mibs{G}{H}$ in the case where $H$ is primitive.
Note that whenever we refer to the ``primitivity'' of a subgroup of $G$, we do so with respect to the natural action of $G$ on $n$ points.
We say that a primitive subgroup $H$ of $G$ is \definedterm{large}, if there are integers $m$ and $k$ such that $H$ is $\left(\symgroup{m} \wr \symgroup{k}\right) \cap G$ in product action on $n = m ^ k$ points or there are integers $m$ and $r$ such that $H$ is $\symgroup{m} \cap \, G$ acting on the $r$-subsets of a set of size $m$, \ie on $n = \binom{m}{r}$ points.
Logarithms are taken to the base $2$.

\begin{theorem}
    \label{theorem:optimal-upper-bounds}
    Suppose $G$ is $\symgroup{n}$ or $\altgroup{n}$ ($n \geqslant 7$) and $H \neq \altgroup{n}$ is a primitive maximal subgroup of $G$.
    \begin{enumerate}[label=(\roman*)]
        \item Either $\mibs{G}{H} < (\log n)^2 + \log n + 1$, or $H$ is large and $\mibs{G}{H} < 3 \sqrt{n} - 1$.
        \item There are infinitely many such $G$ and $H$ for which $\mibs{G}{H} \geqslant \sqrt{n}$.
        \item There are infinitely many such $G$ and $H$ for which $\mibs{G}{H} > \left(\log n\right)^{2} / (2 (\log 3)^2) + \log n / (2 \log 3)$ and $H$ is not large.
    \end{enumerate}
\end{theorem}

We also state our upper bounds for $\mibs{G}{H}$ in terms of $t \coloneqq \cardinality{G : H}$.
It is easy to show that $\mibs{G}{H} \leqslant \basesize{G}{H} \log t$.
Burness, Guralnick, and Saxl showed in \cite{burness_guralnick_saxl_2011} that with finitely many (known) exceptions, in fact $\basesize{G}{H} = 2$, in which case it follows that
\[ \mibs{G}{H} \leqslant 2 \log t. \]
Similar $O(\log t)$ upper bounds on the maximum irredundant base size were recently shown to hold for all non-large-base primitive groups of degree $t$ \cite{gill_lodà_spiga_2022,kelsey_roney-dougal_2022}, raising the question of whether such bounds are best possible in our case.
Using \Cref{theorem:optimal-upper-bounds}, we shall obtain better bounds in terms of $t$.

\begin{corollary}
    \label{corollary:optimal-upper-bounds-degree}
    \begin{enumerate}[label=(\roman*)]
        \item There exist constants $c_1, c_2 \in \positivereal$ such that,
              if $G$ is $\symgroup{n}$ or $\altgroup{n}$ ($n \geqslant 7$) and $H \neq \altgroup{n}$ is a primitive maximal subgroup of $G$ of index $t$,
              then either $\mibs{G}{H} < c_1 (\log \log t)^2$, or $H$ is large and $\mibs{G}{H} < c_2 \left(\log t / \log \log t\right)^{1/2}$.
        \item There is a constant $c_3 \in \positivereal$ and infinitely many such $G$ and $H$ for which $\mibs{G}{H} > c_3 \left(\log t / \log \log t\right)^{1/2}$.
        \item There is a constant $c_4 \in \positivereal$ and infinitely many such $G$ and $H$ for which $\mibs{G}{H} > c_4 (\log \log t)^2$ and $H$ is not large.
    \end{enumerate}
\end{corollary}

\begin{remark}
    \label{remark:main-constants}
    We may take $c_1 = 3.5$, $c_2 = 6.1$, $c_3 = 1$, $c_4 = 0.097$.
    If we assume $n > 100$, then $c_1 = 1.2$ and $c_2 = 4.4$ suffice.
\end{remark}

A sequence $\mathcal{B}$ of points in $\Gamma$ is \definedterm{independent} if no proper subsequence $\mathcal{B}'$ satisfies $G_{(\mathcal{B}')} = G_{(\mathcal{B})}$.
The maximum size of an independent sequence for the action of $G$ on $\Gamma$ is denoted $\height{G}{H}$.
It can be shown that $\basesize{G}{H} \leqslant \height{G}{H} \leqslant \mibs{G}{H}$.
Another closely related property of the action is the \definedterm{relational complexity}, denoted $\relcomp{G}{H}$, a concept which originally arose in model theory.
Cherlin, Martin, and Saracino defined $\relcomp{G}{H}$ in \cite{cherlin_martin_saracino_1996} under the name ``arity'' and showed that $\relcomp{G}{H} \leqslant \height{G}{H} + 1$.

\begin{corollary}
    Suppose $G$ is $\symgroup{n}$ or $\altgroup{n}$ ($n \geqslant 7$) and $H \neq \altgroup{n}$ is a primitive maximal subgroup of $G$.
    Then either $\relcomp{G}{H} < (\log n)^2 + \log n + 2$, or $H$ is large and $\relcomp{G}{H} < 3 \sqrt{n}$.
\end{corollary}

The maximal subgroups of the symmetric and alternating groups were classified in \cite{aschbacher_scott_1985,liebeck_praeger_saxl_1987}.
In order to prove statements (i) and (ii) of Theorem 1, we examine two families of maximal subgroups in more detail and determine lower bounds on the maximum irredundant base size, given in the next two results.

\begin{theorem}
    \label{theorem:affine-case}
    Let $p$ be an odd prime number and $d$ a positive integer such that $p^d \geqslant 7$ and let $n = p^d$.
    Suppose $G$ is $\symgroup{n}$ or $\altgroup{n}$ and $H$ is $\agl{d}{p} \cap G$.
    If $d = 1$, then
    \[ \mibs{G}{H} = 1 + \Omega(p - 1) + \varepsilon(G). \]
    If $d \geqslant 2$ and $p = 3, 5$, then
    \[ \frac{d (d + 1)}{2} + d - 1 + \varepsilon(G) \leqslant \mibs{G}{H} < \frac{d (d + 1)}{2} (1 + \log p) + \varepsilon(G). \]
    If $d \geqslant 2$ and $p \geqslant 7$, then
    \[ \frac{d (d + 1)}{2} + d \, \Omega(p - 1) - 1 + \varepsilon(G) \leqslant \mibs{G}{H} < \frac{d (d + 1)}{2} (1 + \log p) + \varepsilon(G). \]
\end{theorem}

\begin{theorem}
    \label{theorem:wreath-case}
    Let $m \geqslant 5$ and $k \geqslant 2$ be integers and let $n = m^k$.
    Suppose $G$ is $\symgroup{n}$ or $\altgroup{n}$ and $H$ is $(\symgroup{m} \wr \symgroup{k}) \cap G$ in product action.
    Then
    \[ 1 + (m - 1) (k - 1) + \varepsilon(G) \leqslant \mibs{G}{H} \leqslant \frac{3}{2} m k - \frac{1}{2} k - 1. \]
\end{theorem}

After laying out some preliminary results in \Cref{section:preliminaries}, we shall prove \Cref{theorem:affine-case,theorem:wreath-case} in \Cref{section:affine-case} and \Cref{section:wreath-case} respectively, before proving \Cref{theorem:optimal-upper-bounds} and \Cref{corollary:optimal-upper-bounds-degree} in \Cref{section:proof-of-main-result}.

\counterwithin{theorem}{section}

\section{The maximum irredundant base size}
\label{section:preliminaries}

In this section, we collect two general lemmas.
Let $G$ be a finite group acting faithfully and transitively on a set $\Gamma$ with point stabiliser $H$.
If $(\gamma_1, \ldots, \gamma_l)$ is an irredundant base of $G$, then it satisfies \eqref{equation:irredundant-base}.
The tail of the chain in \eqref{equation:irredundant-base} is a strictly descending chain of subgroups in $G_{\gamma_1}$, which is conjugate to $H$.
Therefore,
\[ \mibs{G}{H} \leqslant \length{H} + 1 \leqslant \Omega(\cardinality{H}) + 1. \]
To obtain a lower bound for $\mibs{G}{H}$, one approach is to look for a large explicit irredundant base.
The following lemma says it suffices to find a long chain of subgroups in $G$ such that every subgroup in the chain is a pointwise stabiliser of some subset in $\Gamma$.

\begin{lemma}
    \label{lemma:mibs-is-stabiliser-chain-length}
    Let $l$ be the largest natural number such that there are subsets $\Delta_0, \Delta_1, \ldots, \Delta_l \subseteq \Gamma$ satisfying
    \[ G_{(\Delta_0)} > G_{(\Delta_1)} > \cdots > G_{(\Delta_l)}. \]
    Then $\mibs{G}{H} = l$.
\end{lemma}

\begin{proof}
    Since $l$ is maximal, we may assume that $\Delta_0 = \emptyset$ and $\Delta_l = \Gamma$ and that $\Delta_{i - 1} \subseteq \Delta_i$, replacing $\Delta_i$ with $\Delta_1 \cup \cdots \cup \Delta_i$ if necessary.
    For each $i \in \{1, \ldots, l\}$, write $\Delta_i \setminus \Delta_{i - 1} = \{\gamma_{i, 1}, \ldots, \gamma_{i, m_i}\}$.
    Then $(\gamma_{1, 1}, \ldots, \gamma_{1, m_1}, \allowbreak \gamma_{2, 1}, \ldots, \gamma_{2, m_2}, \allowbreak \ldots, \gamma_{l, 1}, \ldots, \gamma_{l, m_l})$ is a base for $G$ and every subgroup $G_{(\Delta_i)}$ appears in the corresponding chain of point stabilisers.
    Therefore, by removing all redundant points, we obtain an irredundant base of size at least $l$, so $\mibs{G}{H} \geqslant l$.

    On the other hand, given any irredundant base $(\gamma_1, \ldots, \gamma_m)$ of $G$, we can take $\Delta_i \coloneqq \{\gamma_1, \ldots, \gamma_i\}$.
    Therefore, $\mibs{G}{H} = l$.
\end{proof}

Once we have an upper or lower bound for $\mibs{G}{H}$, we can easily obtain a corresponding bound for the maximum irredundant base size of various subgroups of $G$.

\begin{lemma}
    \label{lemma:sym-sandwiches-alt}
    Suppose $M$ is a subgroup of $\symgroup{n}$ with $M \nleqslant \altgroup{n}$.
    Then
    \begin{equation}
        \nonumber
        \mibs{\symgroup{n}}{M} - 1 \leqslant \mibs{\altgroup{n}}{M \cap \altgroup{n}} \leqslant \mibs{\symgroup{n}}{M}.
    \end{equation}
\end{lemma}

\begin{proof}
    This follows immediately from {\cite[Lemma 2.8]{gill_lodà_spiga_2022} and \cite[Lemma 2.3]{kelsey_roney-dougal_2022}}.
\end{proof}

\section{The affine case}
\label{section:affine-case}

In this section, we prove \Cref{theorem:affine-case}.
The upper bounds will follow easily from examinations of group orders.
Therefore, we focus most of our efforts on the construction of an irredundant base, leading to the lower bounds.

Let $p$ be a prime number and $d$ be an integer such that $p^d \geqslant 7$ and let $V$ be a $d$-dimensional vector space over the field $\finitefield{p}$.
Let $G$ be $\symon{V}$ or $\alton{V}$.
Consider the affine group $\aglon{V}$, the group of all invertible affine transformations of $V$, and let $H \coloneqq \aglon{V} \cap G$.

\begin{theorem}[\cite{liebeck_praeger_saxl_1987}]
    \label{theorem:affine-case-maximality}
    The subgroup $H$ is maximal in $G$ (with $p^d \geqslant 7$) if and only if one of the following holds:
    \begin{enumerate}[label=(\roman*)]
        \item $d \geqslant 2$ and $p \geqslant 3$;
        \item $G = \symon{V}$, $d = 1$ and $p \geqslant 7$;
        \item $G = \alton{V}$, $d \geqslant 3$ and $p = 2$;
        \item $G = \alton{V}$, $d = 1$, and $p = 13, 19$ or $p \geqslant 29$.
    \end{enumerate}
\end{theorem}

In this section, we only consider the case where $p$ is odd.
Owing to \Cref{lemma:sym-sandwiches-alt}, we shall assume $G = \symon{V}$ and $H = \aglon{V}$ for now.
In the light of \Cref{lemma:mibs-is-stabiliser-chain-length}, we introduce a subgroup $T$ of diagonal matrices and look for groups containing $T$ that are intersections of $G$-conjugates of $H$ (\Cref{section:subspace-stabilisers}) and subgroups of $T$ that are such intersections (\Cref{section:subgroups-of-diagonal-subgroup}), before finally proving \Cref{theorem:affine-case} (\Cref{section:proof-of-affine-case-theorem}).

\subsection{Subspace stabilisers and the diagonal subgroup}
\label{section:subspace-stabilisers}

Let $T$ be the subgroup of all diagonal matrices in $\glon{V}$ with respect to a basis $\vb_1, \ldots, \vb_d$.
Let $\mu$ be a primitive element of $\finitefield{p}$.
We now find a strictly descending chain of groups from $\symon{V}$ to $T$ consisting of intersections of $G$-conjugates of $H$.
We treat the cases $d = 1$ and $d \geqslant 2$ separately.

\begin{lemma}
    \label{lemma:diagonal-group-as-two-point-stabiliser}
    Suppose $d = 1$ and $G = \symon{V}$.
    Then there exists $x \in G$ such that $H \cap H^{x} = T$.
\end{lemma}

\begin{proof}
    Since $V$ is $1$-dimensional, $\glon{V} = T$ is generated by the scalar multiplication $m_\mu$ by $\mu$.
    Let $\vu \in V \setminus \{ \vzero \}$ and let $t_\vu$ be the translation by $\vu$.
    Then $H = \generatedby{t_\vu} \rtimes \generatedby{m_\mu}$ is the normaliser of $\generatedby{t_\vu}$ in $G$ and $\generatedby{t_\vu}$ is 
    a characteristic subgroup of $H$.
    Hence $H$ is self-normalising in $G$.
    Define (in cycle notation)
    \[
        x \coloneqq (\vu~~\mu^{-1} \vu) (\mu \vu~~\mu^{-2} \vu) \cdots (\mu^{\frac{p - 3}{2}} \vu ~~\mu^{-\frac{p - 1}{2}} \vu) \in G.
    \]
    Then $x \notin H$ and so $x$ does not normalise $H$.
    But $x$ normalises $\generatedby{m_\mu}$, as ${m_\mu}^x = {m_\mu}^{-1}$.
    Therefore,
    \[ T = \generatedby{m_\mu} \leqslant H \cap H^{x} < H. \]
    Since the index $\cardinality{H : T} = p$ is prime, $H \cap H^x = T$.
\end{proof}

The following two lemmas concern the case $d \geqslant 2$.
An affine subspace of $V$ is a subset of the form $\vv + W$, where $\vv \in V$ and $W$ is a vector subspace of $V$.
The (affine) dimension of $\vv + W$ is the linear dimension of $W$.
For an affine transformation $h = g t_\vu$ with $g \in \glon{V}$ and $t_\vu$ denoting the translation by some $\vu \in V$, if $\fix(h)$ is non-empty, then $\fix(h)$ is an affine subspace of $V$, since $\fix(h) = \vv + \ker(g - \id_V)$ for any $\vv \in \fix(h)$.

\begin{lemma}
    \label{lemma:subspace-stabiliser-is-intersection-of-conjugates}
    Suppose $d \geqslant 2$, $p \geqslant 3$, and $G = \symon{V}$.
    Let $W$ be a proper, non-trivial subspace of $V$ and let $K < \glon{V}$ be the setwise stabiliser of $W$.
    Then there exists $x \in G$ such that $H \cap H^x = K$.
\end{lemma}

\begin{proof}
    Let $\lambda \in \finitefield{p}^\times \setminus \{1\}$ and define $x \in \symon{V}$ by setting
    \[
        \vv^x \coloneqq \begin{cases}
            \lambda \vv, & \text{if } \vv \in W, \\
            \vv,         & \text{otherwise}.
        \end{cases}
    \]
    We first show that $K = \centraliser{H}{x}$ and then that $H \cap H^x = \centraliser{H}{x}$. 

    Firstly, let $g \in K$.
    For all $\vv \in W$, we calculate that $\vv^{g^{x}} = (\lambda^{-1} \vv)^{g x} = (\lambda^{-1} \vv^g)^x = \vv^g$.
    For all $\vv \in V \setminus W$, we see that $\vv^{g^{x}} = \vv^{g x} = \vv^g$.
    Hence $g^x = g$, and so $K \leqslant \centraliser{H}{x}$.
    Now, let $h$ be an element of $\centraliser{H}{x}$ and write $h = g t_\vu$ with $g \in \glon{V}$ and $\vu \in V$, so that $h^{-1} = t_{-\vu} g^{-1}$.
    Suppose for a contradiction that there exists $\vv \in W \setminus \{ \vzero \}$ with $\lambda \vv^g + \vu \notin W$.
    Then
    \[ \vv = \vv^{x h x^{-1} h^{-1}} = (\lambda \vv)^{h x^{-1} h^{-1}} = (\lambda \vv^g + \vu)^{x^{-1} h^{-1}} = (\lambda \vv^g + \vu)^{h^{-1}} = \lambda \vv. \]
    Since $\lambda \neq 1$, this is a contradiction and so for all $\vv \in W$,
    \[ \vv = (\lambda \vv^g + \vu)^{x^{-1} h^{-1}} = (\vv^g + \lambda^{-1} \vu)^{h^{-1}} = \vv + (\lambda^{-1} - 1) \vu^{g^{-1}}. \]
    Hence $\vu = \vzero$ and $\vv^g \in W$.
    Therefore, $h = g t_\vzero$ stabilises $W$, whence $h \in K$.
    Thus, $\centraliser{H}{x} = K$.

    Since $\centraliser{H}{x} \leqslant H \cap H^x$, it remains to show that $H \cap H^x \leqslant \centraliser{H}{x}$.
    Suppose otherwise.
    Then there is some $h \in H \cap H^x$ such that $h' \coloneqq x h x^{-1} h^{-1} \neq 1$.
    The set $\fix(h')$ is either empty or an affine subspace of dimension at most $d - 1$.
    Moreover, for any $\vv \in V$, if $\vv \notin (W \setminus \{\vzero\}) \cup W^{h^{-1}}$, then $x$ fixes both $\vv$ and $\vv^h$, and $\vv^{h'} = \vv^{h x^{-1} h^{-1}} = \vv^{h h^{-1}} = \vv$, whence $\vv \in \fix(h')$.
    Therefore,
    \[ V = (W \setminus \{\vzero\}) \cup W^{h^{-1}} \cup \fix(h'). \]
    Then
    \[
        p^d = \cardinality{V} \leqslant \cardinality{W \setminus \{\vzero\}} + \cardinality{W^{h^{-1}}} + \cardinality{\fix(h')} \leqslant (p^{d - 1} - 1) + p^{d - 1} + p^{d - 1} = 3 p^{d - 1} - 1.
    \]
    This is a contradiction as $p \geqslant 3$, and so $H \cap H^x = \centraliser{H}{x} = K$.
\end{proof}

We now construct a long chain of subgroups of $G$ by intersecting subspace stabilisers.

\begin{lemma}
    \label{lemma:diagonal-subgroup-as-intersection-of-subspace-stabilisers}
    Suppose $d \geqslant 2$ and $G = \symon{V}$.
    Let $l_1 \coloneqq d (d + 1) / 2 - 1$.
    Then there exist stabilisers $K_1, \ldots, K_{l_1}$ in $\glon{V}$ of linear subspaces such that
    \begin{equation}
        \label{equation:diagonal-subgroup-as-intersection-of-subspace-stabilisers}
        G > H > K_1 > K_1 \cap K_2 > \cdots > \bigcap_{i = 1}^{l_1} K_i = T.
    \end{equation}
\end{lemma}

\begin{proof}
    Let $\mathcal{I} \coloneqq \{(i, j) \mid i, j \in \{1, \ldots, d\}, i \leqslant j\}\setminus \{ (1, d) \}$ be ordered lexicographically.
    Note that $\cardinality{\mathcal{I}} = l_1$.
    For each $(i, j) \in \mathcal{I}$, let $K_{i, j}$ be the stabiliser in $\glon{V}$ of $\generatedby{\vb_i, \vb_{i + 1} \ldots, \vb_j}$ and define
    $ \mathcal{I}_{i, j} \coloneqq \{ (k, l) \in \mathcal{I} \mid (k, l) \leqslant (i, j) \}. $
    Since $T \leqslant K_{i, j}$ for all $i, j$, we see that
    \begin{equation}
        \nonumber
        T \leqslant \bigcap_{(i, j) \in \mathcal{I}} K_{i, j} \leqslant \bigcap_{i = 1}^{d} K_{i, i} = T.
    \end{equation}
    Hence equality holds, proving the final equality in \eqref{equation:diagonal-subgroup-as-intersection-of-subspace-stabilisers}.

    We now show that, for all $(i, j) \in \mathcal{I}$,
    \[ \bigcap_{\mathclap{(k, l) \in \mathcal{I}_{(i, j)} \setminus \{(i, j)\}}} \; K_{k, l} \;  >  \quad \bigcap_{\mathclap{(k, l) \in \mathcal{I}_{(i, j)}}} \; K_{k, l}. \]
    For $1 \leqslant j < d$, let $g_{1, j}$ be the linear map that sends $\vb_j$ to $\vb_j + \vb_{j + 1}$ and fixes $\vb_k$ for $k \neq j$.
    Then $g_{1, j}$ stabilises $\generatedby{\vb_1}, \ldots, \generatedby{\vb_{j - 1}}$ and any sum of these subspaces, but not $\generatedby{\vb_1, \ldots, \vb_j}$.
    Hence $g_{1, j} \in K_{1, l}$ for all $l < j$ but $g_{1, j} \notin K_{1, j}$.
    For $2 \leqslant i \leqslant j \leqslant d$, let $g_{i, j}$ be the linear map that sends $\vb_j$ to $\vb_{i - 1} + \vb_j$ and fixes $\vb_k$ for $k \neq j$.
    Then $g_{i, j}$ stabilises $\generatedby{\vb_1}, \ldots, \generatedby{\vb_{j - 1}}, \generatedby{\vb_j, \vb_{i - 1}}, \allowbreak \generatedby{\vb_{j + 1}}, \ldots, \generatedby{\vb_d}$ and any sum of these subspaces, but not $\generatedby{\vb_i, \ldots, \vb_j}$.
    Hence $g_{i, j} \in K_{k, l}$ for all $(k, l) < (i, j)$ but $g_{i, j} \notin K_{i, j}$.

    Therefore, the $K_{i, j}$'s, ordered lexicographically by the subscripts, are as required.
\end{proof}

We have now found the initial segment of an irredundant base of $\symon{V}$.
The next subsection extends this to a base.

\subsection{Subgroups of the diagonal subgroup}
\label{section:subgroups-of-diagonal-subgroup}

We now show that, with certain constraints on $p$, every subgroup of $T$ is an intersection of $G$-conjugates of $T$, and hence, by \Cref{lemma:diagonal-subgroup-as-intersection-of-subspace-stabilisers}, an intersection of $G$-conjugates of $H$.
We first prove a useful result about subgroups of the symmetric group generated by a $k$-cycle.

\begin{lemma}
    \label{lemma:cyclic-conjugates-give-any-subgroup}
    Let $s \in \symgroup{m}$ be a cycle of length $k < m$ and let $a$ be a divisor of $k$.
    Suppose that $(k, a) \neq (4, 2)$.
    Then there exists $x \in \symgroup{m}$ such that
    \[ \generatedby{s} \cap \generatedby{s}^x = \generatedby{s^a}. \]
\end{lemma}

\begin{proof}
    Without loss of generality, assume $s = \cycleoneto{k}$ and $a > 1$.
    If $a = k$, then take $x \coloneqq (1~m)$, so that $\generatedby{s} \cap \generatedby{s}^x = 1$, as $m \notin \supp(s^i)$ and $m \in \supp((s^i)^x)$ for all $1 \leqslant i < k$.
    Hence we may assume $a < k$ and $k \neq 4$.
    We find that
    \[ s^a = (1~~a + 1~~\cdots~~k - a + 1) (2~~a + 2~~\cdots~~k - a + 2) \cdots (a~~2a~~\cdots~~k). \]
    Let
    \[ x \coloneqq (1~~2~~\cdots~~a) (a + 1~~a + 2~~\cdots~~2a) \cdots (k - a + 1~~k - a + 2~~\cdots~~k). \]
    Then $\left( s^a \right)^x = s^a$.
    Hence $\generatedby{s^a} = \generatedby{s^a}^x \leqslant \generatedby{s} \cap \generatedby{s}^x$.

    To prove that equality holds, suppose $\generatedby{s^a} < \generatedby{s} \cap \generatedby{s}^x$.
    Then there exists $b \in \{1, \ldots, a - 1\}$ such that $(s^b)^x = s^c$ for some $c$ not divisible by $a$.
    Computing
    \[
        1^{s^c} = 1^{x^{-1} s^b x} = a^{s^b x} = (a + b)^{x} = a + b + 1 = 1^{s^{a + b}}.
    \]
    Therefore,
    \begin{equation}
        \label{equation:cyclic-subgroup-proof-1}
        2^{s^c} = 2^{s^{a + b}} = \begin{cases}
            a + b + 2, & \text{if } b \neq a - 1 \text{ or } k > 2a, \\
            1,         & \text{if } b = a - 1 \text{ and } k = 2a.
        \end{cases}
    \end{equation}
    On the other hand,
    \begin{equation}
        \label{equation:cyclic-subgroup-proof-2}
        2^{x^{-1} s^b x} = 1^{s^b x} = (b + 1)^{x} = \begin{cases}
            b + 2, & \text{if } b \neq a - 1, \\
            1,     & \text{if } b = a - 1.
        \end{cases}
    \end{equation}
    Comparing \eqref{equation:cyclic-subgroup-proof-1} and \eqref{equation:cyclic-subgroup-proof-2}, we see that $b = a - 1$ and $k = 2a$.
    In particular, $a \neq 2$ by the assumption that $k \neq 4$.
    It follows that $a^{s^c} = a^{s^{a + b}} = a - 1$,
    whereas
    \[
        a^{x^{-1} s^b x} = (a - 1)^{s^b x} = (2a - 2)^{x} = 2a - 1,
    \]
    a contradiction.
    The result follows.
\end{proof}

Recall from \Cref{section:subspace-stabilisers} the subgroup $T$ of $\glon{V}$ and the primitive element $\mu$ of $\finitefield{p}$.
For each $i \in \{1, \ldots, d\}$, let $g_i \in \glon{V}$ send $\vb_i$ to $\mu \vb_i$ and fix $\vb_j$ for $j \neq i$.
Then $T = \generatedby{g_1, \ldots, g_d}$.

\begin{lemma}
    \label{lemma:subgroups-of-diagonal-subgroup-as-intersection}
    Suppose $d \geqslant 1$, $p \geqslant 3$, and $G = \symon{V}$.
    Let $i \in \{1, \ldots, d\}$ and let $a$ be a divisor of $(p - 1)$ with $(p, a) \neq (5, 2)$.
    Then there exists $x \in G$ such that
    \[ T \cap T^x = \generatedby{g_1, \ldots, g_{i - 1}, {g_i}^a, g_{i + 1}, \ldots, g_d}. \]
\end{lemma}

\begin{proof}
    Up to a change of basis, $i = 1$.
    The map $g_1 \in \glon{V} < G$ has a cycle $s = (\vb_1~\mu \vb_1~\mu^2 \vb_1 \allowbreak~\cdots~\mu^{p - 2} \vb_1)$.
    Treating $s$ as a permutation on the subspace $\generatedby{\vb_1}$, we see that, for all $\vu \in \generatedby{\vb_1}$ and $\vw \in \generatedby{\vb_2, \ldots, \vb_d}$ (if $d = 1$, then consider $\vw = \vzero$),
    \[ (\vu + \vw)^{g_1} = \vu^{g_1} + \vw = \vu^s + \vw. \]
    By \Cref{lemma:cyclic-conjugates-give-any-subgroup}, since $s$ is a $(p - 1)$-cycle and $(p - 1, a) \neq (4, 2)$, there exists $x \in \symon{\generatedby{\vb_1}}$ such that $\generatedby{s} \cap \generatedby{s}^x = \generatedby{s^a}$.
    Define $\tilde{x} \in G$ by setting
    \[ (\vu + \vw)^{\tilde{x}} \coloneqq \vu^x + \vw \]
    for all $\vu \in \generatedby{\vb_1}$ and $\vw \in \generatedby{\vb_2, \ldots, \vb_d}$.
    Let $g$ be any element of $T$ and write $g = g_1^c g'$ with $c \in \{1, \ldots, p - 1\}$ and $g' \in \generatedby{g_2, \ldots, g_d}$.
    Then, with $\vu, \vw$ as above,
    \[ (\vu + \vw)^g = \vu^{g_1^c} + \vw^{g'} = \vu^{s^c} + \vw^{g'} \]
    and similarly
    \[ (\vu + \vw)^{g^{\tilde{x}}} = \vu^{(s^c)^x} + \vw^{g'}. \]
    Hence $g^{\tilde{x}} \in T$ if and only if $(s^c)^x \in \generatedby{s}$, which holds if and only if $a \mid c$.
    Therefore,
    $ T \cap T^{\tilde{x}} = \generatedby{g_1^a, g_{2}, \ldots, g_d}, $
    as required.
\end{proof}

\begin{lemma}
    \label{lemma:chain-in-diagonal-subgroup}
    Suppose $d \geqslant 1$, $p \geqslant 3$, and $G = \symon{V}$.
    Let $l_2 \coloneqq d$ if $p = 3, 5$, and $l_2 \coloneqq d \, \Omega(p - 1)$ otherwise.
    Then there are subsets $Y_1, \ldots, Y_{l_2} \subseteq G$ such that
    \begin{equation}
        \nonumber
        T > \bigcap_{x \in Y_1} T^x > \bigcap_{x \in Y_2} T^x > \cdots > \bigcap_{x \in Y_{l_2}} T^x = 1.
    \end{equation}
\end{lemma}

\begin{proof}
    First, suppose $p = 3$ or $p = 5$.
    For all $i \in \{1, \ldots, d\}$, by \Cref{lemma:subgroups-of-diagonal-subgroup-as-intersection}, there exists $y_i \in G$ such that
    \[ T \cap T^{y_{i}} = \generatedby{g_1, \ldots, g_{i - 1}, g_{i + 1}, \ldots, g_d}; \]
    setting $Y_i \coloneqq \{y_1, \ldots, y_i\}$ gives
    \[ \bigcap_{x \in Y_i} T^x = \generatedby{g_{i + 1}, \ldots, g_d}. \]
    Therefore, $Y_1, \ldots, Y_d$ are as required.

    Now, suppose $p \geqslant 7$.
    Let $a_1, \ldots, a_{\Omega(p - 1)}$ be a sequence of factors of $(p - 1)$ such that $a_i \mid a_{i + 1}$ for all $i$.
    Let $\mathcal{I} \coloneqq \{1, \ldots, d\} \times \{1, \ldots, \Omega(p - 1) \}$ be ordered lexicographically.
    For each pair $(i, j) \in \mathcal{I}$, by \Cref{lemma:subgroups-of-diagonal-subgroup-as-intersection}, there exists $y_{i, j} \in G$ such that
    \[ T \cap T^{y_{i, j}} = \generatedby{g_1, \ldots, g_{i - 1}, {g_i}^{a_j}, g_{i + 1}, \ldots, g_d}; \]
    setting $Y_{i, j} \coloneqq \{ y_{i', j'} \mid (i', j') \in \mathcal{I}, (i', j') < (i, j) \}$ gives
    \[ \bigcap_{\mathclap{x \in Y_{i, j}}} \, T^x \; = \generatedby{{g_{i}}^{a_j}, g_{i +1 }, \ldots, g_d}. \]
    Therefore, the $Y_{i, j}$'s, ordered lexicographically by the subscripts, are as required.
\end{proof}

This completes our preparations for the proof of \Cref{theorem:affine-case}.

\subsection{Proof of Theorem \ref{theorem:affine-case}}
\label{section:proof-of-affine-case-theorem}

Recall the assumption that $G$ is $\symgroup{p^d}$ or $\altgroup{p^d}$ ($p$ is an odd prime and $p^d \geqslant 7$), which we identify here with $\symon{V}$ or $\alton{V}$, and $H = \agl{d}{p} \cap G$, which we identify with $\aglon{V} \cap G$.

\begin{proof}[Proof of \Cref{theorem:affine-case}]
    First, suppose $d \geqslant 2$, $p \geqslant 3$, and $G = \symon{V}$.
    Let $K_1, \ldots, K_{l_1}$ be as in \Cref{lemma:diagonal-subgroup-as-intersection-of-subspace-stabilisers}.
    For each $i \in \{1, \ldots, l_1\}$, by \Cref{lemma:subspace-stabiliser-is-intersection-of-conjugates}, there exists $x_i \in G$ such that $H \cap H^{x_i} = K_i$.
    Define
    $ X_i \coloneqq \{ 1 \} \cup \{ x_j \mid 1 \leqslant j < i \} \subseteq G $
    for all such $i$.
    Then by \Cref{lemma:diagonal-subgroup-as-intersection-of-subspace-stabilisers},
    \begin{equation}
        \label{equation:affine-case-theorem-1}
        G > H = \bigcap_{x \in X_1} H^x > \bigcap_{x \in X_2} H^x > \cdots > \bigcap_{x \in X_{l_1 + 1}} H^x = T.
    \end{equation}
    Let $Y_1, \ldots, Y_{l_2} \subseteq G$ be as in \Cref{lemma:chain-in-diagonal-subgroup}.
    For each $i \in \{1, \ldots, l_2\}$, let
    $ Z_i \coloneqq \{ x y \mid x \in X_{l_1 + 1}, y \in Y_i \}, $
    so that
    \[ \bigcap_{z \in Z_i} H^z = \bigcap_{y \in Y_i} \left( \bigcap_{x \in X_{l_1 + 1}} H^x \right)^y = \bigcap_{y \in Y_i} T^y. \]
    Then \Cref{lemma:chain-in-diagonal-subgroup} gives
    \begin{equation}
        \label{equation:affine-case-theorem-2}
        T > \bigcap_{z \in Z_{1}} H^x > \bigcap_{z \in Z_{2}} H^x > \cdots > \bigcap_{z \in Z_{l_2}} H^x = 1.
    \end{equation}
    Concatenating the chains \eqref{equation:affine-case-theorem-1} and \eqref{equation:affine-case-theorem-2}, we obtain a chain of length $l_1 + l_2 + 1$.

    Now, suppose $d \geqslant 2$, $p \geqslant 3$, and $G$ is $\symon{V}$ or $\alton{V}$.
    By \Cref{lemma:mibs-is-stabiliser-chain-length} and \Cref{lemma:sym-sandwiches-alt}, since $\aglon{V} \nleqslant \alton{V}$, the lower bounds in the theorem hold.
    For the upper bound on $\mibs{G}{H}$, simply compute
    \begin{align*}
        \mibs{G}{H} & \leqslant 1 + \Omega(\cardinality{H}) \leqslant \Omega(p^d (p^d - 1) (p^d - p) \cdots (p^d - p^{d - 1})) + \varepsilon(G) \\
                    & < \frac{d (d + 1)}{2} + \log((p^d - 1) (p^{d - 1} - 1) \cdots (p - 1)) + \varepsilon(G)                                   \\
                    & < \frac{d (d + 1)}{2} (1 + \log p) + \varepsilon(G).
    \end{align*}

    Finally, suppose $d = 1$ and $p \geqslant 7$.
    Using \Cref{lemma:chain-in-diagonal-subgroup}, we obtain the chain \eqref{equation:affine-case-theorem-2} again.
    Concatenating the chain $G > H > T$ with \eqref{equation:affine-case-theorem-2} and applying \Cref{lemma:mibs-is-stabiliser-chain-length} and \Cref{lemma:sym-sandwiches-alt}, we see that $\mibs{G}{H} \geqslant 1 + \Omega(p - 1) + \varepsilon(G)$.
    In fact, equality holds, as $\mibs{G}{H} \leqslant 1 + \Omega(\cardinality{H}) = 1 + \Omega(p - 1) + \varepsilon(G)$.
\end{proof}

\section{The product action case}
\label{section:wreath-case}

In this section, we prove \Cref{theorem:wreath-case}.
Once again, most work goes into the explicit construction of an irredundant base in order to prove the lower bounds, while the upper bounds will be obtained easily from the length of $\symgroup{n}$.

Throughout this section, let $m \geqslant 5$ and $k \geqslant 2$ be integers, and let $G$ be $\symgroup{m^k}$ or $\altgroup{m^k}$.
Let $M \coloneqq \symgroup{m} \wr \symgroup{k}$ act in product action on
$ \Delta \coloneqq \{ (a_1, \ldots, a_k) \mid a_1, \ldots, a_k \in \{1, \ldots, m\} \} $
and identify $M$ with a subgroup of $\symgroup{m^k}$.

\begin{theorem}[\cite{liebeck_praeger_saxl_1987}]
    \label{theorem:wreath-case-maximality}
    The group $M \cap G$ is a maximal subgroup of $G$ if and only if one of the following holds:
    \begin{enumerate}[label=(\roman*)]
        \item $m \equiv 1 \pmod{2}$;
        \item $G = \symgroup{m^k}$, $m \equiv 2 \pmod 4$ and $k = 2$;
        \item $G = \altgroup{m^k}$, $m \equiv 0 \pmod{4}$ and $k = 2$;
        \item $G = \altgroup{m^k}$, $m \equiv 0 \pmod{2}$ and $k \geqslant 3$.
    \end{enumerate}
\end{theorem}

The strategy to proving the lower bound in \Cref{theorem:wreath-case} is once again to find suitable two-point stabilisers from which a long chain of subgroups can be built.

For each pair of points $\alpha, \beta \in \Delta$, let $\hamming{\alpha}{\beta}$ denote the Hamming distance between $\alpha$ and $\beta$, namely the number of coordinates that differ.

\begin{lemma}
    \label{lemma:hamming-distance}
    Let $x \in M$.
    Then for all $\alpha, \beta \in \Delta$,
    \[ \hamming{\alpha^x}{\beta^x} = \hamming{\alpha}{\beta}. \]
\end{lemma}

\begin{proof}
    Write $x$ as $(v_1, \ldots, v_k) w$ with $v_1, \ldots, v_k \in \symgroup{m}$ and $w \in \symgroup{k}$.
    Let $\alpha = (a_1, \ldots, a_k)$ and $\beta = (b_1, \ldots, b_k)$.
    Write $\alpha^x = (a_1', \ldots, a_k')$ and $\beta^x = (b_1', \ldots, b_k')$.
    Then for each $i \in \{1, \ldots, k\}$,
    \[ a_i = b_i \Longleftrightarrow {a_i}^{v_i} = {b_i}^{v_i} \Longleftrightarrow a_{i^{w}}' = b_{i^{w}}'. \]
    Since $w$ is a permutation of $\{1, \ldots, k\}$, the result holds.
\end{proof}

Define $u \in \symgroup{m}$ to be $\cycleoneto{m}$ if $m$ is odd, and $\cycleoneto{m - 1}$ if $m$ is even, so that $u$ is an even permutation.
Let $U \coloneqq \generatedby{u} \leqslant \symgroup{m}$ and note that $\centraliser{\symgroup{m}}{u} = U$.
The group $U$ will play a central role in the next lemma.

\begin{lemma}
    \label{lemma:wreath-case-two-point-stabiliser}
    Let $i \in \{2, \ldots, k\}$ and $r \in \{1, \ldots, m\}$.
    Let $T_r$ be the stabiliser of $r$ in $\symgroup{m}$ and let $W_i$ be the pointwise stabiliser of $1$ and $i$ in $\symgroup{k}$.
    Then there exists $x_{i, r} \in \altgroup{m^k}$ such that
    \[
        M \cap M^{x_{i, r}} = \left( U \times {(\symgroup{m})}^{i - 2} \times T_r \times {(\symgroup{m})}^{k - i} \right) \rtimes W_i.
    \]
\end{lemma}

\begin{proof}
    Without loss of generality, assume $i = 2$.
    Define $x = x_{2, r} \in \symon{\Delta}$ by
    \[
        (a_1, a_2, \ldots, a_k)^x = \begin{cases}
            ({a_1}^u, a_2, \ldots, a_k) & \text{if } a_2 = r, \\
            (a_1, a_2, \ldots, a_k)     & \text{otherwise}.
        \end{cases}
    \]
    The permutation $x$ is a product of $m^{k - 2}$ disjoint $\orderof{u}$-cycles and is therefore even.

    Let $K \coloneqq \left( U \times T_r \times {(\symgroup{m})}^{k - 2} \right) \rtimes W_2$.
    We show first that $K \leqslant M \cap M^{x}$.
    Let $h = (v_1, \ldots, v_m) w^{-1}$ be an element of $K$.
    Then $v_1 \in U$, $v_2$ fixes $r$, and $w$ fixes $1$ and $2$.
    Therefore, for all $\alpha = (a_1, a_2, \ldots, a_k) \in \Delta$, if $a_2 = r$, then
    \begin{align*}
        \alpha^{h x} & = ({a_1}^{v_1}, a_2, {a_{3^w}}^{v_{3^w}}, \ldots, {a_{k^w}}^{v_{k^w}})^x = ({a_1}^{v_1 u}, a_2, {a_{3^w}}^{v_{3^w}}, \ldots, {a_{k^w}}^{v_{k^w}}) \\
                     & = ({a_1}^{u v_1}, a_2, {a_{3^w}}^{v_{3^w}}, \ldots, {a_{k^w}}^{v_{k^w}}) = ({a_1}^{u}, a_2, a_{3}, \ldots, a_{k})^h = \alpha^{x h};
    \end{align*}
    and if $a_2 \neq r$, then
    \begin{align*}
        \alpha^{h x} & = ({a_1}^{v_1}, {a_2}^{v_2}, {a_{3^w}}^{v_{3^w}}, \ldots, {a_{k^w}}^{v_{k^w}})^x = ({a_1}^{v_1}, {a_2}^{v_2}, {a_{3^w}}^{v_{3^w}}, \ldots, {a_{k^w}}^{v_{k^w}}) \\
                     & = (a_1, a_2, a_3, \ldots, a_k)^h = \alpha^{x h}.
    \end{align*}
    Therefore, $x$ and $h$ commute.
    Since $h$ is arbitrary, $K = K \cap K^x \leqslant M \cap M^x$.

    Let $B$ be the base group $(\symgroup{m})^k$ of $M$.
    Since $K \leqslant M \cap M^x$, we find that $B \cap K \leqslant B \cap M^x$.
    We now show that $B \cap M^x \leqslant B \cap K$, so let $h_1 = (v_1, \ldots, v_k) \in B \cap M^{x}$.
    Then ${h_1}^{x^{-1}} \in M$.
    We show that $v_1 \in U$ and $v_2$ fixes $r$, so that $h_1 \in K$.
    By letting $g_1 \coloneqq (1, 1, v_3, \ldots, v_k) \in K$ and replacing $h_1$ with $g_1^{-1} h_1$, we may assume $v_3 = \cdots = v_k = 1$.
    Let $h_2 \coloneqq x h_1 x^{-1} h_1^{-1} = {h_1}^{x^{-1}} h_1^{-1} \in M$, and let $\alpha \coloneqq (a, b, c, \ldots, c)$ and $\beta \coloneqq (a, r, c, \ldots, c)$ be elements of $\Delta$ with $a \neq m$ and $b \notin \{r, r^{v_2^{-1}}\}$.
    Then $\alpha$ and $\alpha^{h_1}$ are both fixed by $x$, and so $\alpha^{h_2} = \alpha$.
    On the other hand,
    \[
        \beta^{h_2} =
        \begin{cases}
            (a^{u v_1 u^{-1} v_1^{-1}}, r, c, \ldots, c), & \text{if } r^{v_2} = r, \\
            (a^{u}, r, c, \ldots, c),                     & \text{otherwise}.
        \end{cases}
    \]
    Since $\hamming{\alpha^{h_2}}{\beta^{h_2}} = \hamming{\alpha}{\beta} = 1$ by \Cref{lemma:hamming-distance} and $a^u \neq a$, it must be the case that $r^{v_2} = r$ and $a^{u v_1 u^{-1} v_1^{-1}} = a$.
    Therefore, $v_2 \in T_r$ and, as $a$ is arbitrary in $\{1, \ldots, m - 1\}$, we deduce that $v_1 \in \centraliser{\symgroup{m}}{u} = U$ and hence $h_1 \in K$.
    Thus, $B \cap M^x \leqslant B \cap K$ and so $B \cap M^x = B \cap K$.

    To show that $M \cap M^x \leqslant K$, let $h_3 \in M \cap M^x$.
    Now, $B \unlhd M$ and so $B \cap K = B \cap M^x \unlhd M \cap M^x$.
    Therefore,
    \[ h_3 \in \normaliser{M}{B \cap K} = \left( \normaliser{\symgroup{m}}{U} \times T_r \times {(\symgroup{m})}^{k - 2} \right) \rtimes W_2. \]
    The equality uses the fact that $\normaliser{\symgroup{m}}{U} \neq T_r$ (as $m \geqslant 5$).
    Through left multiplication by an element of $K$, we may assume $h_3 \in \normaliser{\symgroup{m}}{U} \times (1_{\symgroup{m}})^{k - 1}$.
    Then $h_3 \in B \cap M^x \leqslant K$.
    Since $h_3$ is arbitrary, $M \cap M^x \leqslant K$.
    Therefore, $K = M \cap M^x$, as required.
\end{proof}

We are now ready to prove the main result for the product action case.
Recall the assumption that $G$ is $\symgroup{m^k}$ or $\altgroup{m^k}$ and $H = M \cap G$.

\begin{proof}[Proof of \Cref{theorem:wreath-case}]
    Firstly, suppose that $H = M$.
    Let $\mathcal{I} \coloneqq \{2, \ldots, k\} \times \{1, \ldots, m - 1\}$, ordered lexicographically.
    For each $(i, r) \in \mathcal{I}$, let $x_{i, r} \in \altgroup{m^k} \leqslant G$ be as in \Cref{lemma:wreath-case-two-point-stabiliser}, and define
    \[ X_{i, r} \coloneqq \{ 1 \} \cup \{ x_{i', r'} \mid (i', r') \in \mathcal{I}, (i', r') \leqslant (i, r) \} \subseteq G. \]
    Then for all $(i, r) \in \mathcal{I}$,
    \[ B \cap \bigcap_{x \in X_{i, r}} M^x = U \times (1_{\symgroup{m}})^{i - 2} \times (\symgroup{m})_{1, \ldots, r} \times (\symgroup{m})^{k - i}. \]
    Hence, for all $(i, r), (j, s) \in \mathcal{I}$ with $(i, r) < (j, s)$,
    $ \bigcap_{x \in X_{i, r}} M^x > \bigcap_{x \in X_{j, s}} M^x. $
    This results in the following chain of stabiliser subgroups, of length $(m - 1)(k - 1) + 2$:
    \[
        G > M > \bigcap_{x \in X_{2, 1}} M^x > \cdots > \bigcap_{x \in X_{2, m - 1}} M^x > \bigcap_{x \in X_{3, 1}} M^x > \cdots > \bigcap_{x \in X_{k, m - 1}} M^x > 1.
    \]
    Therefore, by \Cref{lemma:mibs-is-stabiliser-chain-length}, $\mibs{G}{H} = \mibs{G}{M} \geqslant (m - 1) (k - 1) + 2$.

    Now, if $H \neq M$, then $G = \altgroup{m^k}$, and
    $\mibs{G}{H} \geqslant \mibs{\symgroup{m^k}}{M} - 1 \geqslant (m - 1) (k - 1) + 1$ by \Cref{lemma:sym-sandwiches-alt}.

    Finally, for the upper bound on $\mibs{G}{H}$, we use \eqref{equation:length-symmetric-group-inequality} and \cite[Lemma 2.1]{cameron_solomon_turull_1989} to compute
    \begin{align*}
        \mibs{G}{H} & \leqslant 1 + \length{H} \leqslant 1 + \length{M} \leqslant 1 + k \, \length{\symgroup{m}} + \length{\symgroup{k}}             \\
                    & \leqslant 1 + k \left(\frac{3}{2} m - 2\right) + \left(\frac{3}{2} k - 2\right) \leqslant \frac{3}{2} m k - \frac{1}{2} k - 1.
        \qedhere
    \end{align*}
\end{proof}

\section{Proof of Theorem \ref{theorem:optimal-upper-bounds}}
\label{section:proof-of-main-result}

In this final section, we zoom out for the general case and prove \Cref{theorem:optimal-upper-bounds} by considering the order of $H$ and assembling results from previous sections.

Recall that $G$ is $\symgroup{n}$ or $\altgroup{n}$ ($n \geqslant 7$) and $H \neq \altgroup{n}$ is a primitive maximal subgroup of $G$.
Maróti proved in \cite{maróti_2002} several useful upper bounds on the order of a primitive subgroup of the symmetric group.

\begin{lemma}
    \label{lemma:maróti-bounds}
    \begin{enumerate}[label=(\roman*)]
        \item $\cardinality{H} < 50 n^{\sqrt{n}}$.
        \item At least one of the following holds:
              \begin{enumerate}[label=(\alph*)]
                  \item $H = S_m \cap \, G$ acting on $r$-subsets of $\{1, \ldots, m\}$ with $n = \binom{m}{r}$ for some integers $m, r$ with $m > 2 r \geqslant 4$;
                  \item $H = \left(\symgroup{m} \wr \symgroup{k}\right) \cap G$ with $n = m^k$ for some $m \geqslant 5$ and $k \geqslant 2$;
                  \item $\cardinality{H} < n^{1 + \lfloor \log n \rfloor}$;
                  \item $H$ is one of the Mathieu groups $M_{11}, M_{12}, M_{23}, M_{24}$ acting $4$-transitively.
              \end{enumerate}
    \end{enumerate}
\end{lemma}

\begin{proof}
    (i) follows immediately from \cite[Corollary 1.1]{maróti_2002}.
    (ii) follows from \cite[Theorem 1.1]{maróti_2002} and the description of the maximal subgroups of $\symgroup{n}$ and $\altgroup{n}$ in \cite{liebeck_praeger_saxl_1987}.
\end{proof}

Equipped with these results as well as \Cref{theorem:affine-case,theorem:wreath-case}, we are ready to prove \Cref{theorem:optimal-upper-bounds}.

\begin{proof}[Proof of \Cref{theorem:optimal-upper-bounds}]
    If $H$ is as in case (a) of \Cref{lemma:maróti-bounds}(ii), then $n = \binom{m}{r} \geqslant \binom{m}{2} = \frac{m (m - 1)}{2}$.
    Hence $m < 2 \sqrt{n}$ and, by \eqref{equation:length-symmetric-group-inequality},
    \[ \mibs{G}{H} \leqslant 1 + \length{H} \leqslant 1 + \length{S_m} < 3 \sqrt{n} - 1. \]
    If $H$ is as in case (b) of \Cref{lemma:maróti-bounds}(ii), then $n = m^k$.
    By \Cref{theorem:wreath-case}, $\mibs{G}{H} \leqslant \frac{3}{2} m k - \frac{1}{2} k - 1$.
    If $k = 2$, then
    \[ \mibs{G}{H} \leqslant 3 m - 2 < 3 \sqrt{n} - 1. \]
    If $k \geqslant 3$, then
    \[ \mibs{G}{H} < \frac{3}{2} m \frac{\log n}{\log m} \leqslant \frac{3}{2} \sqrt[3]{n} \frac{\log n}{\log 5}< 3 \sqrt{n} - 1. \]
    If $H$ is as in case (c) of \Cref{lemma:maróti-bounds}(ii), then
    \[ \mibs{G}{H} \leqslant 1 + \length{H} \leqslant 1 + \log \cardinality{H} < 1 + \log \left(n^{1 + \log n}\right) = (\log n)^2 + \log n + 1. \]
    Using the lists of maximal subgroups in \cite{atlas1985}, one can check that $\length{M_{11}} = 7$, $\length{M_{12}} = 8$, $\length{M_{23}} = 11$, and $\length{M_{24}} = 14$.
    It is thus easy to verify that $\mibs{G}{H} \leqslant 1 + \length{H} < (\log n)^2$ in case (d) of \Cref{lemma:maróti-bounds}(ii).
    Therefore, part (i) of the theorem holds.

    We now prove parts (ii) and (iii).
    By \Cref{theorem:affine-case-maximality}, if $n = 3^d$ for some integer $d \geqslant 2$, then $H = \agl{d}{3} \cap G$ is a maximal subgroup of $G$.
    \Cref{theorem:affine-case} now gives
    \[ \mibs{G}{H} > \frac{d^2}{2} + \frac{d}{2} = \frac{(\log n)^2}{2 (\log 3)^2} + \frac{\log n}{2 \log 3}, \]
    as required.

    By \Cref{theorem:wreath-case-maximality}, if $n = m^2$ for some odd integer $m \geqslant 5$, then $H = (\symgroup{m} \wr \symgroup{2}) \cap G$ is a maximal subgroup of $G$.
    \Cref{theorem:wreath-case} now gives $\mibs{G}{H} \geqslant m = \sqrt{n}$, as required.
\end{proof}

Finally, we prove an additional lemma.

\begin{lemma}
    \label{lemma:asymptotics}
    Let $t$ be the index of $H$ in $G$.
    There exist constants $c_5, c_6, c_7, c_8 \in \positivereal$ such that
    \begin{enumerate}[label=(\roman*)]
        \item $c_5 \log t / \log \log t < n < c_6 \log t / \log \log t.$
        \item $c_7 \log \log t < \log n < c_8 \log \log t.$
    \end{enumerate}
\end{lemma}

\begin{proof}
    It suffices to prove that such constants exist for $n$ sufficiently large, so we may assume $n > 100$.
    We first note that $\log t < \log \cardinality{G} \leqslant n \log n$, from which we obtain
    \begin{equation}
        \nonumber
        \log \log t < \log n + \log \log n < \log n + (\log n) \frac{\log \log 100}{\log 100} < 1.412 \log n.
    \end{equation}
    Hence we may take $c_7 = 1 / 1.412 > 0.708$ for $n > 100$.
    By \Cref{lemma:maróti-bounds}(i),
    \begin{align*}
        \log t & = \log \cardinality{G : H} = \log \cardinality{G} - \log \cardinality{H} > \log \frac{n!}{2} - \log \left(50 n^{\sqrt{n}}\right)           \\
               & > \left(n \log n - n \log e - 1\right) - \left(\sqrt{n} \log n + \log 50\right) = n \log n - n \log e - \sqrt{n} \log n - \log 100         \\
               & > n \log n - n (\log e )\frac{\log n}{\log 100} - \sqrt{n} (\log n) \frac{\sqrt{n}}{\sqrt{100}} - (\log 100) \frac{n \log n}{100 \log 100} \\
               & > 0.672 \, n \log n,
    \end{align*}
    where the second inequality follows from Stirling's approximation and the last inequality follows from the fact that $\log e / \log 100 < 0.218$.
    We deduce further that $\log \log t > \log n$ and hence take $c_8 = 1$ for $n > 100$.

    Finally, $\log t / \log \log t < n \log n / \log n = n$ and $\log t / \log \log t > 0.672 \, n \log n / 1.412 \log n = 0.672 \, n / 1.412$.
    Therefore, for $n > 100$, we may take $c_5 = 1$, $c_6 = 1.412 / 0.672 < 2.11$.
\end{proof}

\Cref{corollary:optimal-upper-bounds-degree} now follows by combining \Cref{theorem:optimal-upper-bounds} and \Cref{lemma:asymptotics}.

\begin{remark}
    \label{remark:constants}
    Verifying all cases with $7 \leqslant n \leqslant 100$ by enumerating primitive maximal subgroups of $\symgroup{n}$ and $\altgroup{n}$ in Magma \cite{magma_MR1484478}, we may take $c_5 = 1$, $c_6 = 4.03$, $c_7 = 0.70$, and $c_8 = 1.53$ in the statement of \Cref{lemma:asymptotics}.
    With these values of the constants and those in the proof of \Cref{lemma:asymptotics}, it is straightforward to obtain the values of the constants $c_2, c_3, c_4$ given in \Cref{remark:main-constants}.
    For the values of $c_1$, we use in addition the fact that, for any $n_0$, if $n \geqslant n_0$, then $(\log n)^2 + (\log n) + 1 = (\log n)^2 \left(1 + 1 / \log n + 1 / (\log n)^2 \right) < c_8^2 \left(1 + 1 / \log n_0 + 1 / (\log n_0)^2 \right) (\log \log t)^2.$
\end{remark}

\paragraph{Acknowledgement}
The authors would like to thank the anonymous referees for their careful reading and helpful comments.
The authors would like to thank the Isaac Newton Institute for Mathematical Sciences for its support and hospitality during the programme \textit{Groups, representations and applications: new perspectives}, when work on this article was undertaken.
This work was supported by EPSRC grant N\textsuperscript{\underline{o}} EP/R014604/1, and also partially supported by a grant from the Simons Foundation.
For the purpose of open access, the authors have applied a Creative Commons attribution (CC BY) licence to any Author Accepted Manuscript version arising.

\bibliographystyle{abbrv}

\begin{thebibliography}{10}

    \bibitem{aschbacher_scott_1985}
    M.~Aschbacher and L.~Scott.
    \newblock Maximal subgroups of finite groups.
    \newblock {\em J. Algebra}, 92(1):44--80, 1985.

    \bibitem{magma_MR1484478}
    W.~Bosma, J.~Cannon, and C.~Playoust.
    \newblock The {M}agma algebra system. {I}. {T}he user language.
    \newblock {\em J. Symbolic Comput.}, 24(3-4):235--265, 1997.

    \bibitem{burness_guralnick_saxl_2011}
    T.~C. Burness, R.~M. Guralnick, and J.~Saxl.
    \newblock On base sizes for symmetric groups.
    \newblock {\em Bull. Lond. Math. Soc.}, 43(2):386--391, 2011.

    \bibitem{cameron_solomon_turull_1989}
    P.~J. Cameron, R.~Solomon, and A.~Turull.
    \newblock Chains of subgroups in symmetric groups.
    \newblock {\em J. Algebra}, 127(2):340--352, 1989.

    \bibitem{cherlin_martin_saracino_1996}
    G.~L. Cherlin, G.~A. Martin, and D.~H. Saracino.
    \newblock Arities of permutation groups: Wreath products and \(k\)-sets.
    \newblock {\em J. Comb. Theory Ser. A}, 74(2):249--286, 1996.

    \bibitem{atlas1985}
    J.~H. Conway, R.~T. Curtis, S.~P. Norton, R.~A. Parker, and R.~A. Wilson.
    \newblock {\em An Atlas of Finite Groups}.
    \newblock Oxford University Press, 1985.

    \bibitem{gill_liebeck_2023}
    N.~Gill and M.~W. Liebeck.
    \newblock Irredundant bases for finite groups of {L}ie type.
    \newblock {\em Pac. J. Math.}, 322(2):281--300, may 2023.

    \bibitem{gill_lodà_2022}
    N.~Gill and B.~Lodà.
    \newblock Statistics for \(\operatorname{{S}}_n\) acting on \({k}\)-sets.
    \newblock {\em J. Algebra}, 607:286--299, 2022.

    \bibitem{gill_lodà_spiga_2022}
    N.~Gill, B.~Lodà, and P.~Spiga.
    \newblock On the height and relational complexity of a finite permutation group.
    \newblock {\em Nagoya Math. J.}, 246:372--411, 2022.

    \bibitem{kelsey_roney-dougal_2022}
    V.~Kelsey and C.~M. Roney-Dougal.
    \newblock On relational complexity and base size of finite primitive groups.
    \newblock {\em Pac. J. Math.}, 318(1):89--108, 2022.

    \bibitem{liebeck_praeger_saxl_1987}
    M.~W. Liebeck, C.~E. Praeger, and J.~Saxl.
    \newblock A classification of the maximal subgroups of the finite alternating and symmetric groups.
    \newblock {\em J. Algebra}, 111(2):365--383, 1987.

    \bibitem{maróti_2002}
    A.~Maróti.
    \newblock On the orders of primitive groups.
    \newblock {\em J. Algebra}, 258(2):631--640, 2002.

\end{thebibliography}

\vspace{2em}

\noindent
\begin{minipage}[l]{0.5\textwidth}
    \setlength{\parskip}{0.5em}

    Colva M. Roney-Dougal

    School of Mathematics and Statistics \\
    University of St Andrews \\
    St Andrews \\
    KY16 9SS \\
    United Kingdom

    \href{mailto:Colva.Roney-Dougal@st-andrews.ac.uk}{Colva.Roney-Dougal@st-andrews.ac.uk}

\end{minipage}
\begin{minipage}[r]{0.5\textwidth}
    \setlength{\parskip}{0.5em}

    Peiran Wu

    School of Mathematics and Statistics \\
    University of St Andrews \\
    St Andrews \\
    KY16 9SS \\
    United Kingdom

    \href{mailto:pw72@st-andrews.ac.uk}{pw72@st-andrews.ac.uk}

\end{minipage}

\end{document}